\documentclass[leqno,12pt]{amsart} 
\usepackage[top=30truemm,bottom=30truemm,left=25truemm,right=25truemm]{geometry}
\usepackage{amssymb}
\usepackage{amsmath}
\usepackage{amsthm}
\usepackage{amscd}
\usepackage{mathrsfs}
\usepackage{graphicx}
\usepackage[dvips]{color}
\usepackage[all]{xy}

\usepackage{url}
  \makeatletter

\@addtoreset{equation}{section}
\makeatother
\theoremstyle{plain} 
\newtheorem{theorem}{\indent\bf Theorem}[section]
\newtheorem{lemma}[theorem]{\indent\bf Lemma}
\newtheorem{corollary}[theorem]{\indent\bf Corollary}

\theoremstyle{definition} 
\newtheorem{definition}[theorem]{\indent\bf Definition}
\newtheorem{remark}[theorem]{\indent\bf Remark}

\newcommand{\ddbar}{\partial \bar{\partial}}

\begin{document}
\pagestyle{plain}
\thispagestyle{plain}

\title[pseudonorm]
{Pseudonorms on direct images of pluricanonical bundles}

\author[Takahiro INAYAMA]{Takahiro INAYAMA}
\address{Department of Mathematics,
	Faculty of Science and Technology,
	Tokyo University of Science,
	2641 Yamazaki, Noda,
	Chiba, 278-8510,
	Japan
}
\email{inayama\textunderscore takahiro@rs.tus.ac.jp}
\email{inayama570@gmail.com}
\date{\today}

\begin{abstract}
We study pseudonorms on pluricanonical bundles over Stein manifolds. 
We prove that the pseudonorms determine holomorphic structures of Stein manifolds under certain assumptions. 
This theorem is based on and a generalization of the result obtained by Deng, Wang, Zhang and Zhou \cite{DWZZ} for bounded domains in $\mathbb{C}^n$. 
We also investigate Stein morphisms and the pseudonorms on direct images of pluricanonical bundles. 
Our main goal in this paper is to show that the pseudonorms also determine holomorphic structures of Stein morphisms. 
One important technique is an $L^{2/m}$-variant of the Ohsawa-Takegoshi extension theorem. 
\end{abstract}


\maketitle


\section{Introduction}
The space of holomorphic functions with finite $L^p$-norm $A^p(\Omega)$ on 
a bounded domain $\Omega$, or the $m$-pluricanonical space $H^0(X, mK_X)$ for a complex manifold $X$ 
plays an important role in understanding its geometric property. Initially, Royden proved that 
if $H^0(C, 2K_C)$ is isomorphic to $H^0(C', 2K_{C'})$ with respect to the canonical norm for compact Riemann surfaces $C, C'$ of genus $g\geq 2$, then $C$ is isomorphic to $C'$ \cite{Roy}. 
There are many other generalized results obtained by Markovic \cite{Marko} for more general classes of Riemann surfaces, by Chi and Yau for projective manifolds of general type \cite{Chi}, \cite{ChiYau} and by 
Deng, Wang, Zhang and Zhou for bounded hyperconvex domains in $\mathbb{C}^n$ \cite{DWZZ}. 
In any case, the pseudonorm, called $L^p$-norm or $L^{2/m}$-norm, plays an essential role. 
The program for dealing with projective manifolds of general type via the pseudonorm is called Yau's pseudonorm project (cf. \cite{ChiYau}). 


In this paper, following Deng-Wang-Zhang-Zhou's work \cite{DWZZ}, we study a relatively compact hyperconvex domain in a Stein manifold $X\Subset \widetilde{X}$. 
We prove that the space of pluricanonical forms with the pseudonorm determines a holomorphic structure of the base space. 
To be precise, we obtain the following theorem. 

\begin{theorem}\label{thm:stein}
Let $\widetilde{X}$ be an $n$-dimensional Stein manifold, and $\widetilde{Y}$ be an $l$-dimensional Stein manifold. 
We also let $X\Subset \widetilde{X}$ and $Y\Subset \widetilde{Y}$ be relatively compact hyperconvex domains. 
Assume that there exist $m\geq 2$ and 
a linear isometry 
$$
T : A(X, mK_X) \longrightarrow A(Y, mK_Y)
$$
such that  
$$
\int_X  | u\wedge \overline{u} |^{1/m} = \int_Y | Tu\wedge \overline{Tu} |^{1/m}
$$
for all $u\in A(X, mK_X)$. 
Here we define the space $A(X, mK_X)$ as 
$$
A(X, mK_X):= \{ u\in H^0(X, mK_X) \mid \int_X | u\wedge \overline{u}|^{1/m} < +\infty\},
$$
and 
$$
|u\wedge \overline{u}|^{1/m}:= ( \sqrt{-1}^{mn^2}u\wedge \overline{u} )^{1/m}
$$
for any $u\in H^0(X, mK_X)$ (see Definition \ref{def:l2m}). 

Then we have that $n=l$, and there exists a unique biholomorphic map 
$$
F: X\longrightarrow Y
$$
satisfying the following equation 
$$
|u(z)\wedge \overline{u(z)}|^{1/m}= |F^{\star}(Tu)(z)\wedge \overline{F^{\star}(Tu)(z)}|^{1/m}
$$
for $z\in X$ and $u\in A(X, mK_X)$. 
\end{theorem}

We call a relatively compact domain $D$ in a Stein manifold $\widetilde{X}$ hyperconvex if there exists a negative plurisubharmonic function $\varphi $ on $D$ such that the set $\{ \varphi < c \}$ is relatively compact in $D$ for every $c<0$. 
We can easily see that a hyperconvex domain in $\mathbb{C}^n$ is pseudoconvex. 
We do not have a complete converse. 
However, many pseudoconvex domains become hyperconvex. 
For instance, it was proved that a pseudoconvex domain with Lipschitz boundary is hyperconvex \cite{Dem2}. 


Theorem \ref{thm:stein} is based on and a natural generalization of the theorem obtained by Deng, Wang, Zhang and Zhou \cite{DWZZ} for bounded domains in $\mathbb{C}^n$. 
For any point in Stein manifold, there exist global holomorphic functions that define a coordinate around this point. 
By using this property, we can apply local results obtained by them to Stein cases.

We also investigate Stein manifolds fibered over a complex manifold  $T$ with $\dim T=r$. 
Here we recall a notion of a Stein morphism. 
Let $\widetilde{X}$ be an $r+n$-dimensional complex manifold and $f:\widetilde{X}\to T$ be a holomorphic map. 
In this paper, we say that {\it $f$ is a Stein morphism} if $f$ is a surjective and submersive map with connected fibers, and for any point $t\in T$, there exists an open 
neighborhood $U\subset T$ of $t$ such that $f^{-1}(U)$ is a Stein manifold. 
Then we introduce a notion of a relatively compact Stein morphism. 

\begin{definition}\label{def:relativefibration}
Let $\widetilde{f}:\widetilde{X}\to T$ be a Stein morphism over $T$, and $X\subset \widetilde{X}$ be an open submanifold in $\widetilde{X}$. 
We call $f:=\widetilde{f}|_X:X\to T$ a {\it relatively compact Stein morphism} of $\widetilde{X}$ if the following conditions are satisfied:\\
$({\rm i})$ The map $f$ is a Stein morphism in the above sense. \\
$({\rm ii})$ For each $t\in T$, $X_t:=f^{-1}(t)$ is a relatively compact domain in $\widetilde{X}_t:=\widetilde{f}^{-1}(t)$. 
\end{definition}

In this setting, we prove the following theorem. 

\begin{theorem}\label{thm:fibration}
Let $\widetilde{f}:\widetilde{X}\to B$ and $\widetilde{g}:\widetilde{Y}\to B$ be Stein morphisms over the open unit ball $B\subset \mathbb{C}^r$ with $\dim \widetilde{X}=r+n$ and $\dim \widetilde{Y}=r+l$. 
Suppose that $f=(f_1, \cdots ,f_r):X\to B$ and $g=(g_1, \cdots , g_r):Y\to B$ are relatively compact Stein morphisms of $\widetilde{f}:\widetilde{X}\to B$ and $\widetilde{g}: \widetilde{Y}\to B$, respectively. 
In this local setting, we also let $\widetilde{X}, X, \widetilde{Y}$, and $Y$ be Stein.  
Assume that there exists $m\geq 2$ such that \\
$({\rm i})$ $X_t:=f^{-1}(t)$ and $Y_t:=g^{-1}(t)$ are hyperconvex, and \\
$({\rm ii})$ there exists a linear isomorphism 
$$
T:A(X, mK_{X}) \longrightarrow A(Y, mK_{Y})
$$
such that 
\begin{equation}
\displaystyle \int_{X_t} | U_t\wedge \overline{U_t} |^{1/m} = \int_{Y_t} | (TU)_t \wedge \overline{(TU)_t}|^{1/m} \label{eq:fiberint}
\end{equation}
for any $U\in A(X, mK_{X})$ and $t\in B$, where $U_t \in H^0(X_t, mK_{X_t})$ and $(TU)_t \in H^0(Y_t, mK_{Y_t})$ are uniquely determined $m$-canonical forms such that 
$$
U|_{X_t}=U_t\wedge (df_1\wedge \cdots \wedge df_r)^{\otimes m} , \hspace{5mm} TU|_{Y_t}=(TU)_t \wedge (dg_1\wedge \cdots \wedge dg_r)^{\otimes m}. 
$$
The equation (\ref{eq:fiberint}) includes the case $+\infty = + \infty$. 

Then we have that $T$ induces linear isometries $T_t: A(X_t, mK_{X_t})\to A(Y_t, mK_{Y_t})$, $n=l$, and there exists a unique biholomorphic map 
$F:X\to Y$ such that $f=g \circ F$ and the following equation is satisfied 
$$
|u(z)\wedge \overline{u(z)}|^{1/m}= |F_t^{\star}(T_tu)(z)\wedge \overline{F_t^{\star}(T_tu)(z)}|^{1/m}
$$
for $t\in B$, $z\in X_t$, and $u\in A(X_t, mK_{X_t})$, 
where 
$F_t := F|_{X_t}:X_t \to Y_t$ is a well-defined biholomorphic map. 
\end{theorem}

The fiberwise uniqueness of $F$ implies the following theorem in the global setting. 

\begin{theorem}\label{thm:hariawase}
Let $\widetilde{X}$ and $\widetilde{Y}$ be complex manifolds with $\dim \widetilde{X}=r+n$ and $\dim \widetilde{Y}=r+l$, and 
$\widetilde{f}:\widetilde{X}\to T$ and $\widetilde{g}:\widetilde{Y}\to T$ be Stein morphisms over an $r$-dimensional complex manifold $T$. 
Suppose that $f:X\to T$ and $g:Y\to T$ are relatively compact Stein morphisms of $\widetilde{f}$ and $\widetilde{g}$, respectively. 
Assume that there exists $m\geq 2$ such that \\
$({\rm i})$ $X_t$ and $Y_t$ are hyperconvex, and \\
$({\rm ii})$ there exists an isomorphism of sheaves 
$$
T: f_\star(mK_{X/T}) \to g_\star(mK_{Y/T})
$$
which satisfies the following conditions: \\
For any open set $D\subset T$, there is a linear isomorphism 
$$
T_D: f_\star(mK_{X/T})(D)\to g_\star(mK_{Y/T})(D)
$$
such that 
\begin{equation}
\displaystyle \int_{X_t} | U_t\wedge \overline{U_t} |^{1/m} = \int_{Y_t} | (T_DU)_t \wedge \overline{(T_DU)_t}|^{1/m} \label{eq:globalint}
\end{equation}
for any $U\in f_\star(mK_{X/T})(D)$ and $t\in D$.
The equation (\ref{eq:globalint}) includes that case $+\infty = +\infty$. 

Then we have that $T$ induces linear isometries $T_t: A(X_t, mK_{X_t})\to A(Y_t, mK_{Y_t})$, $n=l$, and there exists a unique biholomorphic map $F:X\to Y$ such that $f=g\circ F$ and the following equation is satisfied 
$$
|u(z)\wedge \overline{u(z)}|^{1/m}= |F_t^{\star}(T_tu)(z)\wedge \overline{F_t^{\star}(T_tu)(z)}|^{1/m}
$$
for $t\in D$, $z\in X_t$, and $u\in A(X_t, mK_{X_t})$, 
where 
$F_t := F|_{X_t}:X_t \to Y_t$ is a well-defined biholomorphic map. 

\end{theorem}

Theorems \ref{thm:fibration} and \ref{thm:hariawase} say that the $L^{2/m}$-norms on direct images of pluricanonical bundles determine holomorphic structures of fibrations. 
The so-called $m$-th Narasimhan-Simha Hermitian metric on direct images of relative pluricanonical bundles has been studied by several people (cf. \cite{BP08}, \cite{HPS}, \cite{PT}).
The above theorems also demonstrate the importance of the $m$-th pseudonorms on them. 

For bounded domains in $\mathbb{C}^n$, 
we obtain the following corollary. 

\begin{corollary}\label{cor:domainfibration}
Let $X\Subset \mathbb{C}^r_t \times \mathbb{C}^n_z$ and 
$Y\Subset \mathbb{C}^r_t\times \mathbb{C}^l_w $ be bounded pseudoconvex domains fibered over the open unit ball $B \subset \mathbb{C}^r$. 
We also let $f:X\to \mathbb{C}^r$ and $g:Y\to \mathbb{C}^r$ be natural projections such that $f(t, z)=t$ and $g(s, w)=s$ with $f(X)=g(Y)=B$. 
Assume that there exists $0<p<2$ such that \\
$({\rm i})$ $X_t:=f^{-1}(t)$ and $Y_t:=g^{-1}(t)$ are hyperconvex domains for each $t\in B$, \\
$({\rm ii})$ there exists a linear isomorphism 
$$
T : A^p(X)\longrightarrow A^p(Y)
$$
with 
$$
\int_{X_t}| \Phi|_{X_t} |^p d\mu_n = \int_{Y_t}|(T\Phi)|_{Y_t} |^p d\mu_l
$$
for any $\Phi \in A^p(X)$ and $t\in B$, where $d\mu_n$ and $d\mu_l$ are the standard Lebesgue measures on $\mathbb{C}^n$ and $\mathbb{C}^l$, respectively. 
Then we have that $T$ induces linear isometries $T_t: A^p(X_t)\to A^p(Y_t)$, $n=l$, and there exists a unique biholomorphic map 
$F:X\to Y$ such that $f=g \circ F$ and the following equation is satisfied
$$
|\phi(z)|=|T_t\phi(F_t(z))||J_{F_t}(z)|^{2/p}
$$ 
for $t\in B$, $z\in X_t$, and $\phi \in A^p(X_t)$, where $F_t:=F|_{X_t}:X_t\to Y_t$ is a well-defined biholomorphic map and $J_{F_t}$ is the holomorphic Jacobian of $F_t$.  
\end{corollary}

Theorem \ref{thm:fibration} holds only for $p=2/m$, whereas we can prove Corollary \ref{cor:domainfibration} for all $0<p<2$ by using the same argument of the proof of Theorem \ref{thm:fibration}. 
Corollary \ref{cor:domainfibration} is a relative version of Theorem 1.2 in \cite{DWZZ}. 
A key proposition to prove the main theorems is an $L^{2/m}$-variant of the Ohsawa-Takegoshi extension theorem. 

On the other hand, we can prove Corollary \ref{cor:domainfibration} without using the $L^{2/m}$-variant of the Ohsawa-Takegoshi extension theorem. 
We will show the proof in Section \ref{app:direct}. 

The organization of this paper is as follows. In Section \ref{sec:prelim}, we introduce some definitions and properties of $L^{2/m}$-norms and $m$-th Bergman kernels. 
In Section \ref{sec:absolute}, we give a proof of Theorem \ref{thm:stein}. 
In Section \ref{sec:relative}, we prove Theorems \ref{thm:fibration}, \ref{thm:hariawase}, and Corollary \ref{cor:domainfibration}.
At last, in Section \ref{app:direct}, we show a simple proof of Corollary \ref{cor:domainfibration}.

\vskip10mm
{\bf Acknowledgment. }
The author would like to thank Prof. Shigeharu Takayama for enormous supports. 
He is also grateful to Prof. Fusheng Deng and Prof. Takeo Ohsawa for answering questions and helpful comments. 
He also wishes to express his gratitude to the anonymous referee for their suggestions to improve the manuscript. 
The author is supported by Japan Society for the Promotion of Science, Grant-in-Aid for Research Activity Start-up (Grant No. 21K20336).

\section{Preliminaries}\label{sec:prelim}
\subsection{$L^{2/m}$-norm}

We prepare some basic definitions and properties to show the main theorem. 
Throughout this section, we denote by $X$ an $n$-dimensional complex manifold, and by $K_X$ the canonical line bundle over $X$.   

Firstly, we confirm the following notation. 

\begin{definition}[$L^{2/m}$-norm]\label{def:l2m}
We take a local coordinate $\{U, (z_1, \cdots, z_n )\}$ on $X$. 
For a holomorphic section $u\in H^0(X, mK_X)$, we locally define $(\sqrt{-1}^{mn^2}u\wedge \overline{u})^{1/m}$ as 
$$
(\sqrt{-1}^{mn^2}u\wedge \overline{u})^{1/m}= |f_U|^{2/m}\sqrt{-1}^{n^2}(dz_1\wedge\cdots \wedge dz_n) \wedge (d\overline{z}_1 \wedge \cdots \wedge d\overline{z}_n), 
$$
where $u= f_U(z_1, \cdots , z_n) (dz_1\wedge\cdots \wedge dz_n)^{\otimes m}$ on $U$. 
We can verify that $(\sqrt{-1}^{mn^2}u\wedge \overline{u})^{1/m}$ gives a globally defined real non-negative $(n, n)$-form. 
For simplicity, we define $|u\wedge \overline{u}|^{1/m} := (\sqrt{-1}^{mn^2}u\wedge \overline{u})^{1/m}$. 
Then, we can define the $L^{2/m}$-norm $\| \cdot \|_{X, m}$ of $u$ as 
$$
\| u\|_{X, m}= \int_X |u\wedge \overline{u}|^{1/m}. 
$$
\end{definition}


For $m>2$, $\| \cdot \|_{X, m}$ become only pseudonorms, i.e. they satisfy the norm axioms except the homogeneity. 
However, we call them $L^{2/m}$-norms for all $m\in\mathbb{N}$. 

We remark that the space $A(X, mK_X)$ might be $\{ 0\}$ even though $X$ is a Stein manifold. 
Hence, in this paper, we mainly consider the space $A(X, mK_X)$ over a relatively compact Stein domain in Stein manifold. 
If $X$ is a relatively compact Stein domain in some Stein manifold $\widetilde{X}$, $A(X, mK_X)$ is an infinite-dimensional vector space and has the separation property since $H^0(\widetilde{X}, mK_{\widetilde{X}})|_X \subset A(X, mK_X)$ and $\widetilde{X}$ is Stein. 
Here the separation property means that for any points $x\neq y\in X$, there exist sections $\sigma_1, \sigma_2 \in A(X, mK_X)$ such that $\sigma_1(x)=0, \sigma_1(y)\neq 0, \sigma_2(x)\neq 0, \sigma_2(y)=0$. 
We also know that $A(X, mK_X)$ is a complete separable metric space with respect to the metric $d(u_1, u_2):= \| u_1-u_2\|_{X, m}$ for $m\geq 2$.


A fundamental lemma to prove the main theorems is the following result about isometries between $L^p$-spaces. 

\begin{theorem}$($\cite{Rudin}$)$\label{thm:rudin}
Let $\mu$ and $\nu$ be finite positive measures on two sets $U$ and $V$. 
We also let $p\in \mathbb{R}_{>0}$ be not even, and $N$ be a positive integer. 
If $\{ f_i\}_{1\leq i \leq N} \subset L^p(U, \mu), \{g_i\}_{1\leq i \leq N}\subset L^p(V, \nu)$  satisfy 
$$
\int_U |1+\sum_{1\leq i \leq N}\alpha_i f_i|^p d\mu= \int_V |1+ \sum_{1\leq i \leq N}\alpha_i g_i|^p d\nu 
$$
for all $(\alpha_1, \cdots, \alpha_N)\in \mathbb{C}^N$, then $(f_1, \cdots, f_N)$ and $(g_1, \cdots, g_N)$ are called equimeasurable, i.e. 
for every real-valued non-negative Borel function $u:\mathbb{C}^N\to \mathbb{R}_{\geq 0}$, we get 
$$
\int_U u(f_1, \cdots , f_N) d\mu = \int_V u(g_1, \cdots , g_N) d\nu .
$$
Moreover, let $I:X\to \mathbb{C}^N$ and $J:Y\to \mathbb{C}^N$ be the maps $I=(f_1, \cdots, f_N)$ and $J=(g_1, \cdots, g_N)$. Then we obtain 
$$
\mu(I^{-1}(E))=\nu(J^{-1}(E))
$$
for every Borel set $E\subset \mathbb{C}^N$. 
\end{theorem}

This theorem implies the next lemma, which is necessary for us to prove the main theorems. This is a version of Lemma 2.1 in \cite{Marko}, and of Lemma 2.2 in \cite{DWZZ}. 

\begin{lemma}\label{lem:rudin2}
Let $X$ and $Y$ be two relatively compact domains in an $n$-dimensional Stein manifold $\widetilde{X}$ and an $l$-dimensional Stein manifold $\widetilde{Y}$, respectively. 
Suppose that $\{ u_k\}_{k=0}^{N} \subset A(X, mK_X), \{v_k\}_{k=0}^N \subset A(Y, mK_Y)$ are pluricanonical sections such that for every $N$-tuple of complex numbers $\{ \alpha_k\}_{k=1}^N$, we have 
$$
\int_X  |(u_0+\sum_{1\leq k \leq N}\alpha_k u_k)\wedge (\overline{u_0+\sum_{1\leq k \leq N} \alpha_k u_k})|^{1/m}
$$
$$
=\int_Y |(v_0+\sum_{1\leq k \leq N}\alpha_k v_k)\wedge (\overline{v_0+\sum_{1\leq k \leq N} \alpha_k v_k})|^{1/m}. 
$$
If neither $u_0$ nor $v_0$ is constantly zero, then for every real-valued non-negative Borel function $f:\mathbb{C}^N\to \mathbb{R}_{\geq 0}$, we have 
$$
\int_X f\left( \frac{u_1}{u_0}, \cdots , \frac{u_N}{u_0}\right) |u_0 \wedge \overline{u}_0|^{1/m} = \int_Y f\left( \frac{v_1}{v_0}, \cdots , \frac{v_N}{v_0}\right) |v_0 \wedge \overline{v}_0|^{1/m}. 
$$
Here we regard $u_k /u_0$ and $v_k / v_0$ as functions on $X$ and $Y$, respectively. 
\end{lemma} 

\begin{proof}
Set $d\mu=|u_0 \wedge \overline{u}_0|^{1/m}$ and $d\nu=|v_0 \wedge \overline{v}_0|^{1/m}$. 
The measures $\mu$ and $\nu$ are well-defined on $X$ and $Y$, respectively. Then we have that 
$u_k/u_0 \in L^{2/m}(X, d\mu)$ and $v_k / v_0 \in L^{2/m}(Y, d\nu)$, respectively. From the assumption, we have that 
$$
\int_X | 1+ \sum_{1\leq k \leq N}\alpha_k \frac{u_k}{u_0} |^{2/m} d\mu = \int_Y | 1+ \sum_{1\leq k \leq N}\alpha_k \frac{v_k}{v_0} |^{2/m} d\nu
$$
for every $N$-tuple of complex numbers $\{ \alpha_k\}_{k=1}^N$. 
Therefore Lemma \ref{lem:rudin2} follows from Theorem \ref{thm:rudin}. 
\end{proof}

We introduce the following result obtained by Deng, Wang, Zhang, and Zhou. 
This is a fundamental result and the absolute version of Corollary \ref{cor:domainfibration}. 
\begin{theorem}$($\cite[Theorem 1.2]{DWZZ}$)$\label{thm:DWZZ}
Let $\Omega_1\subset \mathbb{C}^n$ and $\Omega_2 \subset \mathbb{C}^m$ be bounded hyperconvex domains. 
If there is a linear isometry $T:A^p(\Omega_1)\to A^p(\Omega_2)$ for some $0<p<2$, then $n=m$ and 
there exists a unique biholomorphic map $F:\Omega_1 \to \Omega_2$ such that 
$$
|T\phi\circ F||J_F|^{2/p}=|\phi|
$$
for all $\phi \in A^p(\Omega_1)$, where $J_F$ is the holomorphic Jacobian of $F$. 
\end{theorem}

\subsection{$m$-th Bergman kernel}
In this subsection, we let $\widetilde{X}$ denote an $n$-dimensional Stein manifold, and $X$ denote a relatively compact Stein domain in $\widetilde{X}$. 
Firstly, we introduce the definition of the $m$-th Bergman kernel $K_{X,m}$ and the exhaustivity of it. 

\begin{definition}[$m$-th Bergman kernel]\label{def:mbergman}
We set 
$$
K_{X, m}(z):= \sup \{ |u\wedge \overline{u}|^{1/m}(z) \mid u\in A(X, mK_X), \int_X |u\wedge \overline{u}|^{1/m} \leq 1\}, 
$$
for each point $z\in X$. 

We also say that {\it $K_{X, m}$ is  exhaustive} if the following function 
$$
K'_{X, m}(z):=\sup \{ |u\wedge \overline{u}|^{1/m}/dV_{\widetilde{X}}(z) \mid u\in A(X, mK_X), \int_X |u\wedge \overline{u}|^{1/m} \leq 1 \}
$$
is exhaustive on $X$ for a volume form $dV_{\widetilde{X}}$ on $\widetilde{X}$. 
\end{definition}

\begin{remark}
The function $K'_{X, m}$ depends on the choice of volume forms on $\widetilde{X}$. 
However, the exhaustivity of $K_{X, m}$ is independent of them. 
\end{remark}

Since $\widetilde{X}$ is Stein, we can take a volume form $dV_{\widetilde{X}}$ (= smooth Hermitian metric on $-K_{\widetilde{X}}$) with curvature positive on $X$. 
Taking this metric, we have that $K'_{X, m}$ is a continuous plurisubharmonic function on $X$ (cf. \cite[Lemma 6.2 and 6.3]{DWZZ1}, \cite[Proposition 28.3]{HPS}, \cite{PT}). 
Hence, the exhaustivity of $K_{X, m}$ implies the pseudoconvexity of $X$. 

On the other hand, if $\widetilde{X}=\mathbb{C}^n$ and $X=\Omega$ is a bounded domain in $\mathbb{C}^n$, it is known \cite{NZZ} that the pseudoconvexity of $\Omega$ implies the exhaustivity of $K_{\Omega, m}$ for $m\geq 2$. 

\begin{theorem}$($\cite[Theorem 2.7]{NZZ}$)$\label{thm:nzz}
Let $\Omega \Subset \mathbb{C}^n$ be any bounded domain. Then $\Omega$ is pseudoconvex if and only if $K_{\Omega, p}$ is an exhaustion function for $p\in (0, 2)$. 
\end{theorem}

We also prove that $K_{X, m}$ is exhaustive when $\widetilde{X}$ is a Stein manifold and $X\Subset \widetilde{X}$ is a relatively compact hyperconvex domain. 
To prove this theorem, we prepare the following results. 
The first one is a localization principle. 
This is obtained by Ohsawa for bounded pseudoconvex domains in $\mathbb{C}^n$ \cite{Oh1}. 
A more general result appears in \cite{Oh2}. 
We can prove this principle by using H\"ormander's $L^2$-estimate. 

\begin{lemma}\label{lem:localization}
In the above setting, we let $a\in \partial X$ be a boundary point.
Then there exists an open neighborhood $U_0$ of $ a $ such that 
for any two open neighborhoods $V\Subset U \subset U_0$  of $a$,  
there is a positive constant $C$ such that 
$$
K_{U\cap X}(x)\leq CK_{X}(x)
$$
for any $x\in V\cap X$. Here $K_{U\cap X}$ and $K_{X}$ are Bergman kernels of $U\cap X$ and $X$, respectively,  
and $C$ is independent of $x$. 
\end{lemma}

\begin{proof}[\indent \sc proof of lemma \ref{lem:localization}]
	We fix a K\"ahler form $\widetilde{\omega}$ on $\widetilde{X}$, and set $\omega := \widetilde{\omega}|_{X}$. 
	Since $ \widetilde{X} $ is a Stein manifold, we can take global holomorphic functions $(g_1, \ldots , g_n)\in \mathcal{O}(\widetilde{X})^{n}$ and an open neighborhood $U_0$ of $a$ such that $(g_1, \ldots , g_n)$ defines a biholomorphic coordinate map on $U_0$. 
	We will modify the norms of $g_i$ such that 
	$$
	\sup_{X}|g_i|\leq \frac{1}{2\sqrt{n}}
	$$ for $1\leq i \leq n$.
	We also take a smooth strictly plurisubharmonic function $\psi $ on $\widetilde{X}$ such that 
	$\psi <0$ and 
	$$
	\sqrt{-1}\ddbar \psi \geq \omega 
	$$
	on $ X $. 
	Then $\phi(z):=(n+1)\log (|g_1(z)-g_1(x)|^2+\cdots |g_n(z)-g_n(x)|^2) + \psi(z) $ satisfies $\phi <0$ and 
	$$
	\sqrt{-1}\ddbar \phi \geq \omega 
	$$
	on $X$. 
	
	Suppose that $V$ and $U$ are open neighborhoods of $a$ with $V \Subset U \subset U_0$. 
	Let $x\in V\cap X$ be any point. 
	We take a cut-off function $\chi \in C^\infty_c(U)$ which satisfies $0\leq \chi \leq 1$ and $\chi = 1$ on a neighborhood of $\overline{V}$. We have a holomorphic $(n, 0)$-form $f$ on $U\cap X$ such that $|f\wedge \overline{f}|(x)=K_{U\cap X}(x)$ and 
	$$
	\int_{U\cap X}|f\wedge \overline{f}|=1. 
	$$
	We define an $(n, 1)$-form $\alpha :=\overline{\partial}(\chi f)$ on $X$. We get 
	$$
	\int_X |\alpha|^2_{\omega}e^{-\phi}dV_{\omega} < +\infty.
	$$
	Thanks to H\"ormander's $L^2$-estimate (cf. \cite{DemCom}), we can obtain a solution $u$ satisfying $\overline{\partial} u=\alpha$ and 
	$$
	\int_X |u\wedge \overline{u}|e^{-\phi} \leq \int_X |\alpha|^2_{\omega}e^{-\phi}dV_{\omega}. 
	$$
	Let $\beta:=\chi f-u$. Then $\beta$ is a holomorphic $(n, 0)$-form on $X$, $|\beta \wedge \overline{\beta}|(x)= |f\wedge \overline{f}|(x)$, and 
    \begin{align*}
    \left (\int_X |\beta \wedge \overline{\beta}|\right)^{1/2} & \leq 1+ \left( \int_X |u\wedge \overline{u}|^2 e^{-\phi}\right)^{1/2} \\
    & \leq 1+ \left( \int_{supp|\overline{\partial}\chi |} |\alpha|^2_{\omega} e^{-\phi} dV_{\omega} \right)^{1/2} \\
    & <C. 
    \end{align*}
	 For the reason that 
	 $$
	 \inf_{z\in supp|\overline{\partial}\chi |} \inf_{x\in V\cap X} (|g_1(z)-g_1(x)|^2+\cdots |g_n(z)-g_n(x)|^2)>0,
	 $$
	 the above constant $C$ is independent of $x$. Therefore, we have 
	 $$
	 K_{U\cap X}(x)\leq  \frac{C^2|\beta \wedge \overline{\beta}|(x)}{\int_X |\beta \wedge \overline{\beta}|},
	 $$
	 which completes the proof. 
	
\end{proof}

The second one is the exhaustivity of the Bergman kernels of bounded hyperconvex domains. 
This result was obtained by Ohsawa. 
\begin{theorem}$($\cite{Oh3}$)$\label{thm:hyperconvex}
Let $D$ be a bounded hyperconvex domain in $\mathbb{C}^n$. Then $\lim_{z\to \partial D}K_D(z)=+\infty$. 
\end{theorem}

Combining the above results, we can prove the following theorem. 

\begin{theorem}\label{thm:exhaustive}
In the above setting, we have that $\lim_{z\to \partial X} K_{X, m}(z)=+\infty$, that is, the $m$-th Bergman kernel is exhaustive. 
\end{theorem}

\begin{proof}[\indent \sc proof of theorem \ref{thm:exhaustive}]
For any point $a\in \partial X$, we take an open coordinate neighborhood $U_0$ and two open hyperconvex neighborhoods $V \Subset U\subset U_0$ of $a$ which 
satisfy the condition of Lemma \ref{lem:localization}. 
We have $K_{U\cap X}(z)\leq CK_{X}(z)$ for $z\in V\cap X$ and some positive constant $C>0$. 
Then Theorem \ref{thm:hyperconvex} implies that $\lim_{z\to a} K_{X}(z)=+\infty$. 
We also obtain $K_{X}(z)\leq K_{X, m}(z)$ since for any holomorphic  $ (n,0) $-form $ u\in A(X, K_X) $, $ u^{\otimes m} \in A(X, mK_X) $ and 
$$
|u^{\otimes m} \wedge \overline{u}^{\otimes m}|^{1/m}=|u \wedge \overline{u}|.
$$
Hence, we obtain the conclusion. 

\end{proof}

\section{Proof of Theorem \ref{thm:stein}}\label{sec:absolute}
In this section, we prove Theorem \ref{thm:stein}. The main argument of the proof of this theorem is almost the same as the proof of Theorem 1.2 in \cite{DWZZ}. 
Before proving Theorem \ref{thm:stein}, we provide basic settings and some lemmas. 

We set $H_{X,z}:=\{ u\in A(X, mK_X)\mid u(z)=0\}$ to be a hyperplane in $A(X, mK_X)$. 
First, we define a subset $X'$ of $X$ as follows. We say that $z\in X'$ if and only if there exists $w\in Y$ such that $T(H_{X,z})=H_{Y, w}$.  
The separation property of $A(X, mK_X)$ implies that there exists a unique $w\in Y$ such that the equation $T(H_{X,z})=H_{Y, w}$ holds. 
Therefore, we can define a map $F:X'\to Y$ by setting $F(z)=w$ if $T(H_{X,z})=H_{Y, w}$. Set $Y':= F(X')$. Then $F$ is a bijection from $X'$ to $Y'$. 
Here we take $u_0\in A(X, mK_X)$ and a countable dense subset $\{ u_1, u_2, \cdots \}$ of $A(X, mK_X)$ with $u_0 \neq 0$. Then $\{ v_0:= Tu_0, v_1:= Tu_1, \cdots \}$ is a countable dense subset of $A(Y, mK_Y)$. 
We define maps $I_N, J_N, I_\infty$, and $J_\infty$ as follows. 
$$
\begin{array}{cccc}
I_N :& X & \longrightarrow & \mathbb{C}^{N} \\
& \rotatebox{90}{$\in $} & & \rotatebox{90}{$\in $} \\
& z & \longmapsto & \left( \frac{u_1(z)}{u_0(z)}, \cdots , \frac{u_N(z)}{u_0(z)}  \right) \\
& & & \\
J_N :& Y & \longrightarrow & \mathbb{C}^{N} \\
& \rotatebox{90}{$\in $} & & \rotatebox{90}{$\in $} \\
& w & \longmapsto & \left( \frac{v_1(w)}{v_0(w)}, \cdots , \frac{v_N(w)}{v_0(w)}  \right) \\
& & & \\
\end{array}
$$

$$
\begin{array}{cccc}
I_\infty :& X & \longrightarrow & \mathbb{C}^{\infty} \\
& \rotatebox{90}{$\in $} & & \rotatebox{90}{$\in $} \\
& z & \longmapsto & \left( \frac{u_1(z)}{u_0(z)},  \frac{u_2(z)}{u_0(z)}, \cdots   \right) \\
& & & \\
J_\infty :& Y & \longrightarrow & \mathbb{C}^{\infty} \\
& \rotatebox{90}{$\in $} & & \rotatebox{90}{$\in $} \\
& w & \longmapsto & \left( \frac{v_1(w)}{v_0(w)}, \frac{v_2(w)}{v_0(w)}, \cdots  \right).
\end{array}
$$

The maps are well-defined on $X\setminus u_0^{-1}(0)$ and $Y\setminus v_0^{-1}(0)$, respectively. 
We also obtain that $X'\setminus u_0^{-1}(0) = \cap_{N}I_N^{-1}J_N(Y\setminus v_0^{-1}(0))=I^{-1}_\infty J_\infty(Y\setminus v_0^{-1}(0))$ and $Y'\setminus v_0^{-1}(0)=\cap_{N}J_N^{-1}I_N(X\setminus u_0^{-1}(0))=J_\infty^{-1}I_\infty(X\setminus u_0^{-1}(0)) $. 
The separation property of $A(X, mK_X)$ and $A(Y, mK_Y)$ implies that $I_\infty$ and $J_\infty$ are injective. 
For $z\in I_\infty^{-1}J_\infty(Y\setminus v_0^{-1}(0))$, we get $F(z)=J_\infty^{-1}(I_\infty(z))$. 

Lemma \ref{lem:rudin2} implies the following lemma. 

\begin{lemma}$($cf. \cite[Lemma 2.4]{DWZZ}$)$\label{lem:zeroset}
The measures of $X\setminus X'$ and $Y \setminus Y'$ are zero with respect to any smooth positive volume forms on $X$ and $Y$, respectively. 
\end{lemma}

We fix smooth positive volume forms on $\widetilde{X}$ and $\widetilde{Y}$. 
Let $V\subset Y\setminus v_0^{-1}(0)$ be a local open coordinate such that $V\cap (Y'\setminus v_0^{-1}(0))$ has positive measure in $Y$. 
Since $\widetilde{Y}$ is a Stein manifold, there exist global holomorphic functions $\{ g_j\}_{j=1}^l$ on $Y$ such that $(w_1:=g_1|_V, \cdots , w_l:=g_l|_V)$ defines a coordinate. 
We choose $u_0, u_1, \cdots \subset A(X, mK_X)$ such that $Tu_0=v_0, Tu_1=g_1v_0, \cdots , Tu_l=g_lv_0$. 
Then $v_j/v_0=g_j$ for $1\leq j \leq l$. Set $V':=V\cap (Y'\setminus v_0^{-1}(0)),  U':= F^{-1}(V')\subset X'\setminus u_0^{-1}(0)$. 
Since $V'=J_\infty^{-1}(I_\infty(U'))\subset J_l^{-1}(I_l(U'))$, $J_l^{-1}(I_l(U'))$ also has positive measure in $Y$. 
Then $I_l(U')$ has positive measure in $\mathbb{C}^l$ for the reason that $J_l|_{V'}$ is a biholomorphic coordinate function on $V'$. 
Note that $I_l$ is a non-constant holomorphic map on $X\setminus u_0^{-1}(0)$, and $I_l(X\setminus v_0^{-1}(0))$ also has positive measure in $\mathbb{C}^l$. Then we have $n\geq l$. 
The same argument implies that $l\geq n$. Hence, we get the following lemma. 

\begin{lemma}
The dimensions of $X$ and $Y$ are equal, i.e. $n=l$. 
\end{lemma}

We can also prove the following lemma. 

\begin{lemma}\label{lem:open}
The sets $X'$ and $Y'$ are open in $X$ and $Y$, respectively. 
\end{lemma}

\begin{proof}[\indent \sc proof of lemma \ref{lem:open}]
We take points $z_0\in X'$ and $F(z_0)=:w_0\in Y'$. 
We take a local open coordinate $V\subset Y$ around $w_0$, global holomorphic functions $(g_1, \cdots , g_n)$ on $Y$ which define a local coordinate on $V$, and $v_0\in A(Y, mK_Y)$ such that $v_0 \neq 0$ on $V$. 
We choose a countable dense set $u_0, u_1, \cdots \subset A(X, mK_X)$ such that $Tu_0=v_0, Tu_1=g_1v_0, \cdots , Tu_n=g_nv_0$. 
Since $I_n:X\setminus u_0^{-1}(0) \to \mathbb{C}^n$ is holomorphic, $I_n^{-1}(J_n(V))=:U$ is an open set in $X\setminus u_0^{-1}(0)$ around $z_0$. 

Set $U':=U\cap X'$. It follows that 
\begin{equation*}
F=J_\infty^{-1}\circ I_\infty = J_n^{-1}\circ I_n \label{eq:hol}
\end{equation*}
on $U'$. Therefore, we have 
\begin{equation}
\frac{u(z)}{u_0(z)} = \frac{Tu(J_n^{-1}\circ I_n(z))}{Tu_0(J_n^{-1}\circ I_n(z))} \label{eq:frac}
\end{equation}
for all $u\in A(X, mK_X)$ and $z\in U'$. Lemma \ref{lem:zeroset} implies that $U'$ is dense in $U$. 
By continuity of  $J_n^{-1}\circ I_n$, the equation (\ref{eq:frac}) holds on $U$. 
Then we get $U\subset X'$. Since $U$ is also open in $X$ and $z_0 \in U$, we have $X'$ is open in $X$. 
The same argument implies that $Y'$ is open in $Y$. 
\end{proof}

The above argument also implies that $F$ is holomorphic on $X'$, i.e. $F$ is a biholomorphic map from $X'$ to $Y'$. 

Theorem \ref{thm:rudin} and Lemma \ref{lem:rudin2} imply the following result. 

\begin{lemma}$($cf. \cite[Lemma 2.7]{DWZZ}$)$\label{lem:unique}
For any $z\in X'$ and $u\in A(X, mK_X)$, we have 
$$
|u(z)\wedge \overline{u(z)}|^{1/m}= |F^{\star}(Tu)(z)\wedge \overline{F^{\star}(Tu)(z)}|^{1/m}
$$
\end{lemma}

We will show that the biholomorphic map $F:X'\to Y'$ can be extended to a biholomorphic map from $X$ to $Y$. 
Before proving it, we give the following result. 




\begin{lemma}\label{lem:pluripolar}
Let $S:=X\setminus X'$. 
Then $S$ is a closed pluripolar set, i.e. for any point $a \in S$, there exist an open neighborhood $U$ of $a$ and a plurisubharmonic function $\rho $ on $U$ such that $S\cap U\subset \rho^{-1}(-\infty)$. 
\end{lemma}

\begin{proof}[\indent \sc proof of lemma \ref{lem:pluripolar}]
We take 
any convergent sequence $\{ a_j\}\subset X\setminus S \to a\in S$. 
Since $Y$ is relatively compact, by passing to a subsequence, we can assume that there is $b\in \overline{Y}$ such that $F(a_j)\to b$. 
If $b\in Y$, we have $a\in X'$ from equation \ref{eq:frac} or the definition of $F$. Therefore $b\in \partial Y$. 

We take open coordinates $U\Subset \widetilde{U}_{\{z_1, \cdots , z_n\}}\subset \widetilde{X}$ around $a$ and $V\Subset \widetilde{V}_{\{w_1, \cdots , w_n\}}\subset \widetilde{Y}$ around $b$. 
In this local setting, the equation of Lemma \ref{lem:unique} gives us the following expression 
$$
|g_u(z)|=|h_{Tu}(F(z))|\left| \frac{\partial (F_1, \cdots , F_n)}{\partial (z_1, \cdots , z_n)}(z)\right|^m 
$$
for $z\in X'$ and any $u\in A(X, mK_X)$. Here $F_j=w_j\circ F$, $u=g_u (dz_1 \wedge \cdots \wedge dz_n)^{\otimes m}$, and $Tu=h_{Tu}(dw_1\wedge \cdots \wedge dw_n)^{\otimes m}$. 
The  exhaustivity of $K_{Y, m}$ implies that $K_{Y, m}(F(a_j))\to +\infty $. 
The choices of coordinates imply that boundary behaviors of $K_{X, m}/(\sqrt{-1}dz_1\wedge d\overline{z}_1 \wedge \cdots \sqrt{-1}dz_n\wedge d\overline{z}_n)$ and $K_{Y, m}/(\sqrt{-1}dw_1\wedge d\overline{w}_1 \wedge \cdots \sqrt{-1}dw_n\wedge d\overline{w}_n)$ coincide with $K'_{X, m}$ and $K'_{Y, m}$ on $U\cap X$ and $V\cap Y$, respectively. 
Then, for each $j$, we have $v_j \in A(Y, mK_Y)$ such that $\| v_j \|_{Y, m}=1$ and $|h_{v_j}(F(a_j))| \to +\infty$ as $j \to +\infty$, where $v_j = h_{v_j}(dw_1\wedge \cdots \wedge dw_n)^{\otimes m}$. 
Set $u_j := T^{-1}(v_j)$. Then we have 
$$
|g_{u_j}(a_j)|=|h_{v_j}(F(a_j))|\left| \frac{\partial (F_1, \cdots , F_n)}{\partial (z_1, \cdots , z_n)}(a_j)\right|^m, 
$$
and $\| u_j \|_{X, m}=1$. Since $K_{X, m}$ is locally bounded on $X$, the left-hand side has an upper bound. Then it follows that 
$$
\left| \frac{\partial (F_1, \cdots , F_n)}{\partial (z_1, \cdots , z_n)}(a_j)\right| \to 0.
$$

We define a function $\rho : U \to \mathbb{R}\cup \{ -\infty\}$ as 
$$
\rho (z)= \begin{cases}
\log \left| \frac{\partial (F_1, \cdots , F_n)}{\partial (z_1, \cdots , z_n)}(z)\right| & (z\in U\setminus S) \\
-\infty & (z\in U\cap S). 
\end{cases}
$$
This function satisfies the mean-value inequality. The above argument ensures that $\rho$ is upper semi-continuous on $U$ for the following reason. 
If $\limsup_{z\to a} |\rho (z)|> -\infty $ for some point $a\in S$, passing to a subsequence, we have that $K_{X, m}(z)\to +\infty$ as $z\to a$, which is a contradiction. 
Hence, $\rho$ is a plurisubharmonic function on $U$. By definition, we see that $S \cap U \subset \rho^{-1}(-\infty)$. 

\end{proof}

By using the following fact, we can prove Theorem \ref{thm:stein}. 

\begin{theorem}$($cf. \cite[Theorem 5.24, Corollary 5.25]{DemCom}$)$\label{thm:demailly}
Let $A\subset X$ be a closed pluripolar set in a complex analytic manifold $X$. 
Then \\
$({\rm i})$ every plurisubharmonic function $v$ on $X\setminus A$ that is locally bounded above near $A$ extends uniquely into a function $\widetilde{v}$ on $X$, and  \\
$({\rm ii})$ every holomorphic function $f$ on $X\setminus A$ that is locally bounded near $A$ extends to a holomorphic function on $X$. 
\end{theorem}

\begin{proof}[\indent \sc proof of theorem \ref{thm:stein}]
Taking a local coordinate or embedding $Y$ into a bounded domain in the complex Euclidean space, we can regard $F$ as a bounded map. 
Since $S$ is a closed pluripolar set, 
there exists a holomorphic function $\widetilde{F}$ on $X$ such that $\widetilde{F}|_{X\setminus S}=F$. 
We also denote by $F$ this extension. 
The hyperconvexity of $Y$ implies the existence of a negative plurisubharmonic function $\varphi $ on $Y$ such that $\{ w\in Y \mid \varphi(w)<c\}$ is relatively compact for any $c<0$. 
Then $\widetilde{\varphi}:=\varphi \circ F$ is also negative plurisubharmonic function on $X\setminus S$, and can be extended to a plurisubharmonic function on $X$ by Theorem \ref{thm:demailly}. 
Hence, $\widetilde{\varphi}$ attains its maximum on $S$ for the reason that $F(S)\subset \overline{Y}\setminus Y$. 
By the maximum principle, $\varphi$ must be a constant function, which is a contradiction. Then $S=\emptyset $. 

By applying the same method to $Y$, we obtain $Y\setminus Y'=\emptyset$. 
Consequently, $F$ is globally defined on $X$ and a biholomorphic map from $X$ to $Y$. 
\end{proof}


We can also prove the uniqueness of $F$. The equation 
$$
|u(z)\wedge \overline{u(z)}|^{1/m}= |F^{\star}(Tu)(z)\wedge \overline{F^{\star}(Tu)(z)}|^{1/m}
$$
for any $u\in A(X, mK_X)$ implies that the condition $u(z)=0$ is equivalent to $Tu(F(z))=0$. 
Namely, $F$ is uniquely determined by $T$ as $T(H_{X, z})=H_{Y, F(z)}$.

\section{Proof of Theorem \ref{thm:fibration}}\label{sec:relative}
In this section, we prove Theorems \ref{thm:fibration} and \ref{thm:hariawase}. 
A key ingredient to prove them is the following $L^{2/m}$-variant of the Ohsawa-Takegoshi extension theorem. 

\begin{theorem}$($cf. \cite{BP}, \cite{GZ}, \cite{HPS}, \cite{PT}$)$\label{thm:lpohsawa}
Let $X$, $B$, $f$, and the notation be the same as in Theorem \ref{thm:fibration}. 
Then for any $t\in B$ and $u\in A(X_t, mK_{X_t})$, there exists an extension $U\in A(X, mK_X)$ such that 
$U|_{X_t}=u\wedge (df_1 \wedge \cdots \wedge df_r)^{\otimes m}$ and 
$$
\int_X |U\wedge \overline{U}|^{1/m} \leq C \int_{X_t} |u\wedge \overline{u}|^{1/m}, 
$$
for positive constant $C>0$ which is independent of $t$, $u$, and $U$. 
\end{theorem}

\begin{lemma}\label{lemma:fiberwise}
The linear isomorphism $T$ in Theorem \ref{thm:fibration} induces fiberwise linear isometries $\{ T_t\}_{t\in B}$, i.e.  
$$
T_t : A(X_t, mK_{X_t}) \longrightarrow A(Y_t, mK_{Y_t})
$$
is a linear isomorphism and 
$$
\int_{X_t} |u\wedge \overline{u} |^{1/m} = \int_{Y_t} | T_tu \wedge \overline{T_tu}|^{1/m} 
$$
for all $t\in B$ and $u \in A(X_t, mK_{X_t})$. 
\end{lemma}

\begin{proof}[\indent \sc proof of lemma \ref{lemma:fiberwise}]
We define the well-defined maps $\{ T_t\}_{t\in B}$. Let $u \in A(X_t, mK_{X_t})$. 
We can take an extension $U \in A(X, mK_X)$ such that $U|_{X_t}=u\wedge (df_1\wedge \cdots \wedge df_r)^{\otimes m}$ and 
$$
\int_{X}|U\wedge \overline{U}|^{1/m} \leq C \int_{X_t} |u\wedge \overline{u}|^{1/m}
$$
for some positive constant $C>0$ from Theorem \ref{thm:lpohsawa}. Then we define $T_t(u):=(TU)_t$, where $(TU)_t\in A(Y_t, mK_{Y_t})$ and $(TU)|_{Y_t}=(TU)_t \wedge (df_1\wedge \cdots \wedge df_r)^{\otimes m}$. We have to show that $T_t$ are well-defined. 
If $U_1, U_2 \in A(X, mK_X)$ are both extensions of $u$ satisfying the above properties, 
$(U_1-U_2)_{t}=0$ and 
\begin{align*}
&\int_{X_t} |(U_1-U_2)_{t}\wedge \overline{(U_1-U_2)_{t}}|^{1/m} \\
=&\int_{Y_t}| (TU_1-TU_2)_t \wedge \overline{(TU_1-TU_2)_t}|^{1/m}\\
=&0.
\end{align*}
Therefore we get  $(TU_1)_t = (TU_2)_t$. 

We also have  
\begin{align*}
\int_{X_t} |u\wedge \overline{u} |^{1/m} &= \int_{X_t} |U_t\wedge \overline{U_t} |^{1/m} \\
&= \int_{Y_t} | (TU)_t \wedge \overline{(TU)_t}|^{1/m}\\
&= \int_{Y_t} | T_t(u) \wedge \overline{T_t(u)}|^{1/m}. 
\end{align*}
Similarly, we can prove that $T_t$ is surjective. Hence, $T_t : A(X_t, mK_{X_t}) \to A(Y_t, mK_{Y_t})$ is a linear isometry. 
\end{proof}

\begin{proof}[\indent \sc Proof of Theorem \ref{thm:fibration}]

We can construct biholomorphic maps $\{ F_t:X_t\to Y_t\}_{t\in B}$ induced by linear isometries $\{ T_t\}_{t\in B}$ by Theorem \ref{thm:stein}. 
Hence, we get $n=l$. 
Then we will make a global holomorphic map $F$ from $X$ to $Y$. 
We define a map $F$ as follows: 
$$
\begin{array}{ccc}
X & \stackrel{F}{\longrightarrow} & Y \\
\rotatebox{90}{$\in $} & & \rotatebox{90}{$\in$} \\
z & \longmapsto & F_{f(z)}(z). \\
\end{array}
$$
We know that $F_t(z)$ is holomorphic in the fiber directions for each fixed $t$. It is sufficient to show that 
the map $F$ is holomorphic in all the directions. 

We take $U_0\in A(X, mK_X)$ and a countable dense subset $\{ U_1, U_2,  \cdots \}\subset A(X, mK_X)$. Then 
$\{V_0:=TU_0, V_1:=TU_1, \cdots \}$ is a countable dense subset of $A(Y, mK_Y)$. 
We remark that Theorem \ref{thm:lpohsawa} ensures that $A(X, mK_X)$ is infinite-dimensional. 
We define maps $I_N, J_N, I_\infty, J_\infty$ as follows: 
$$
\begin{array}{cccc}
I_N :& X\setminus U_0^{-1}(0) & \longrightarrow & B\times \mathbb{C}^{N} \\
& \rotatebox{90}{$\in $} & & \rotatebox{90}{$\in $} \\
& z & \longmapsto & (f(z), (\frac{U_1( z)}{U_0( z)}, \cdots, \frac{U_N(z)}{U_0(z)} )) \\
& & & \\
J_N :& Y\setminus V_0^{-1}(0) & \longrightarrow & B\times \mathbb{C}^{N} \\
& \rotatebox{90}{$\in $} & & \rotatebox{90}{$\in $} \\
& w & \longmapsto & (g(w), (\frac{V_1(w)}{V_0(w)}, \cdots, \frac{V_N(w)}{V_0(w)} ))\\
& & & \\
I_\infty :& X\setminus U_0^{-1}(0) & \longrightarrow & B\times \mathbb{C}^{\infty} \\
& \rotatebox{90}{$\in $} & & \rotatebox{90}{$\in $} \\
& z & \longmapsto & (f(z), (\frac{U_1( z)}{U_0( z)}, \frac{U_2(z)}{U_0(z)}, \cdots )) \\
& & & \\
J_\infty :& Y\setminus V_0^{-1}(0) & \longrightarrow & B\times \mathbb{C}^{\infty} \\
& \rotatebox{90}{$\in $} & & \rotatebox{90}{$\in $} \\
& w & \longmapsto & (g(w), (\frac{V_1(w)}{V_0(w)}, \frac{V_2(w)}{V_0(w)}, \cdots )).
\end{array}
$$

Next, we show that $F=J_\infty^{-1}\circ I_\infty$ on $I_\infty^{-1}\circ J_\infty(Y\setminus V_0^{-1}(0))$. 
A separation property of $A(X_t, mK_{X_t})$ and $A(Y_t, mK_{Y_t})$ implies that $I_\infty$ and $J_\infty$ are injective maps. 
For any $u \in A(X_t, mK_{X_t})$, we get an extension $U \in A(X, mK_X)$ of $u$ such that $U_t =u$. 
Since $\{ U_j \}$ (resp. $\{ TU_j \}$) is a dense subset of $A(X, mK_X)$ (resp. $A(Y, mK_Y)$), we can take a sequence $\{ U_{j_k}\} \subset \{ U_j \}$ 
such that $\| U_{j_k} -U \|_{X, m}\to 0$ and $\| TU_{j_k}-TU \|_{Y, m}\to 0$ as $k\to +\infty$. 
Therefore, we have that $U_{j_k}$ (resp. $TU_{j_k}$) 
converges compactly to $U$ (resp. $TU$) on $X$ (resp. $Y$). 
Since all compact sets of $X_t$ (resp. $Y_t$) are compact in $X$ (resp. $Y$), $(U_{j_k})_t$ (resp. $(TU_{j_k})_t$) also converges compactly to $u$ (resp. $T_tu$) on $X_t$ (resp. $Y_t$). 
Here, we remark that we do not know whether $(U_{j_k})_t\in A(X_t, mK_{X_t})$ (resp. $(TU_{j_k})_t\in A(Y_t, mK_{Y_t})$). 

For any point $z\in I_\infty^{-1}\circ J_\infty(Y\setminus V_0^{-1}(0))$, there exists a unique point $w\in Y\setminus V_0^{-1}(0)$ such that 
$$
(f(z), (\frac{U_1(z)}{U_0(z)}, \frac{U_2(z)}{U_0(z)}, \cdots ))  = (g(w), (\frac{V_1(w)}{V_0(w)}, \frac{V_2(w)}{V_0(w)}, \cdots )). 
$$
Therefore, $f(z)=g(w)$ and 
$$
\frac{U_j(z)}{U_0(z)}=\frac{V_j(w)}{V_0(w)}
$$
for all $j\in \mathbb{N}$. Set $t:=f(z)=g(w)$. 
The above argument implies that for any $u \in A(X_t, mK_{X_t})$, we have 
$$
\frac{u(z)}{(U_0)_t(z)}=\frac{T_tu(w)}{(V_0)_t(w)}.
$$
We have that $u(z)=0$ if and only if $T_tu(w)=0$. 
Hence, this $w$ satisfies the following equations 
$$
H_{Y_t, w}=T_t(H_{X_t, z}) 
$$
and $F_t(z)=w$. Consequently, we obtain that 
$$
J_\infty^{-1}(I_\infty(z))=F_{f(z)}(z)=F(z)
$$
for any $z\in I_\infty^{-1}\circ J_\infty(Y\setminus V_0^{-1}(0))$. 

%

Then
we show that the map $F$ is holomorphic on $X$. 
It is enough to show that $F$ is holomorphic around any point $z_0\in X$ locally. 
We set $w_0:=F(z_0)$ and $f(z_0)=g(w_0)=t_0$.
Since $g=(g_1, \cdots , g_r)$ defines a submersive map,  we can take global holomorphic functions $(g_1, \cdots , g_{r}, \eta_{r+1}, \cdots, \eta_{r+n})\in \mathscr{O}(\widetilde{Y})^{n+r}$ which define a coordinate around $w_0$. 
Fix $V_0\in A(Y, mK_Y)$ such that $V_0(w_0)\neq 0$ and $(V_0)_{t_0}\in A(Y_{t_0}, mK_{Y_{t_0}})$. 
We take such an open coordinate $W\Subset W'' \subset Y$ around $w_0$ such that $V_0|_W\neq 0$. 
We also define $V_j = \eta_{j+r}V_0$ for $j\in \{ 1, 2, \cdots , n\}$.
Taking an open neighborhood $B'\subset B\subset \mathbb{C}^r$ of $t_0$, we have $V_j\in A(g^{-1}(B'), mK_Y)$ for all $j$. 
Since the setting is local, we may regard $B'=B$ and $V_j\in A(Y, mK_Y)$. 

We take $Z:=I_n^{-1}(J_n(W)) \subset X\setminus U_0^{-1}(0)$. 
By the constructions of $V_0$ and $J_n$, we have that $Z$ is an open neighborhood of $z_0$. 
Note that $I_n(z_0)=J_n(w_0)$ since $F_{t_0}(z_0)=w_0$ and $(U_j)_{t_0}\in A(X_{t_0}, mK_{X_{t_0}}), (V_j)_{t_0}\in A(Y_{t_0}, mK_{Y_{t_0}})$ for $1\leq j\leq n$. 
We also set $Z'=Z\cap I_\infty^{-1}( J_\infty(Y\setminus V_0^{-1}(0)))$. 
The above argument implies that $F(z)=J_\infty^{-1}(I_\infty(z))=J_n^{-1}(I_n(z))$ for any $z\in Z'$. 
Hence, it follows that 
\begin{equation}\label{eq:renzo}
\frac{U(z)}{U_0(z)}=\frac{TU(J_n^{-1}\circ I_n(z))}{V_0(J_n^{-1}\circ I_n(z))}
\end{equation}
for any $z\in Z'$ and $U\in A(X, mK_X)$. 
By the continuity of $J_n^{-1}\circ I_n$, the equation (\ref{eq:renzo}) holds on $\overline{Z'}$.

Here we remark that, thanks to Fubini's theorem, there exists a set of Lebesgue measure zero $N\subset B$  such that $(U_j)_t\in A(X_t, mK_{X_t})$ for $t\in B\setminus N$ and $j\in \mathbb{N}$. 
For these $t\in B\setminus N$, 
we can construct the maps $I_{\infty,t}: X_t\to \mathbb{C}^\infty$ and $J_{\infty,t}: Y_t\to \mathbb{C}^\infty$ on each fiber.
By definition, it holds that 
\begin{align*}
	I_\infty^{-1}&(J_\infty(Y\setminus V_0^{-1}(0)))\cap X_t\\
	= \{& z\in X_t\setminus (U_0)_t^{-1}(0) \mid \exists! w\in Y_t\setminus (V_0)_t^{-1}(0) \text{ such that } f(z)=g(w)=t~ \& \\
	&\left( \frac{(U_1)_t(z)}{(U_0)_t(z)}, \frac{(U_2)_t(z)}{(U_0)_t(z)} ,\cdots \right) =\left( \frac{(V_1)_t(w)}{(V_0)_t(w)}, \frac{(V_2)_t(w)}{(V_0)_t(w)} ,\cdots \right) \}.
\end{align*}
The family $\{ (U_j)_t\}_{j\in \mathbb{N}}$ might not be a dense subset in $A(X_t, mK_{X_t})$ with respect to $\|~\cdot~\|_{X_t, m}$, but $\{ (U_j)_t\}_{j\in \mathbb{N}}$ has the following property: for any $u\in A(X_t, mK_{X_t})$, there exists a subsequence $\{ U_{j_k}\}_{k\in \mathbb{N}}\subset \{U_j\}_{j\in \mathbb{N}}$ (resp. $\{ V_{j_k}\}_{k\in \mathbb{N}}\subset \{V_j\}_{j\in \mathbb{N}}$) such that 
$\{ (U_{j_k})_t\}_{k\in \mathbb{N}}$ (resp. $\{ (V_{j_k})_t\}_{k\in \mathbb{N}}$) converges compactly to $u$ (resp. $T_tu$) in $X_t$ (resp. $Y_t$).
Indeed, due to Theorem \ref{thm:lpohsawa}, we can take an extension $U\in A(X, mK_X)$ of $u$ such that $U|_{X_t}=u\wedge (df_1\wedge \cdots \wedge df_r)^{\otimes m}$ and 
$$
\int_X |U\wedge \overline{U}|^{1/m}\leq C \int_{X_t} |u\wedge \overline{u}|^{1/m}. 
$$
Since $\{ U_j\}_{j\in \mathbb{N}}$ is a dense subset in $A(X, mK_X)$, 
we can get a subsequence $\{ U_{j_k}\} \subset \{ U_j\}_{j\in \mathbb{N}}$ such that 
$\| U_{j_k}-U\|_{X, m}\to 0$, 
which also implies that  $U_{j_k}$ converges compactly to $U$. 
Then we can conclude that $(U_{j_k})_t$ (resp. $(TU_{j_k})_t$) also converges to $u$ (resp. $T_tu$) on every compact subset in $X_t$ (resp. $Y_t$). 
Taking the limit, we obtain that
\begin{align*}
&I_\infty^{-1}(J_\infty(Y\setminus V_0^{-1}(0)))\cap X_t\\
&= \{ z\in X_t\setminus (U_0)_t^{-1}(0) \mid \exists! w\in Y_t\setminus (V_0)_t^{-1}(0) \text{ such that } \frac{u(z)}{(U_0)_t(z)}=\frac{T_tu(w)}{(TU_0)_t(w)}~ \forall u\in A(X_t, mK_{X_t}) \}
\end{align*}
for $t\in B\setminus N$. 
In the previous section, we denote by $X'_t$ the set of the right side, and prove that $X'_t=X_t$. 
Therefore, we see that $Z\setminus Z'$ is of measure zero. 
Then, we can conclude that the equation (\ref{eq:renzo}) holds on $Z$. 
Thus $Z\subset I_\infty^{-1}(J_\infty(Y\setminus V_0^{-1}(0)))$.

On $Z$, we have that 
$$
F(z)=(f_1(z), \cdots , f_r(z), (\frac{U_1(z)}{U_0(z)}, \cdots, \frac{U_{n}(z)}{U_0(z)})), 
$$
where $U_j=T^{-1}(V_j)$ for $j\in \{ 1, 2, \cdots , n\}$. 
Hence, we can say that $F$ is holomorphic on $Z$, which is an open neighborhood of an arbitrary point $z_0$. 
Then $F:X\to Y$ is a biholomorphic map and satisfies $f=g\circ F$. 


By the construction of $F$, we also have that $F|_{X_t}=F_t$. Hence, we get 
$$
|u(z)\wedge \overline{u(z)}|^{1/m}= |F_t^{\star}(T_tu)(z)\wedge \overline{F_t^{\star}(T_tu)(z)}|^{1/m}
$$
for $t\in B$, $z\in X_t$, and $u\in A(X, mK_X)$. 
\end{proof}

\begin{remark}

For each $t\in B$, we can construct the map $I_{\infty, t}: A(X_t, mK_{X_t})\to A(Y_t, mK_{Y_t})$ as in Section \ref{sec:absolute}. 
Then $\{ I_{\infty, t}\}_{t\in B}$ consists of uncountable elements. 
However, by using the Ohsawa-Takegoshi extension theorem properly, 
we can choose countably infinite elements $I_\infty=(U_1/U_0, U_2/U_0, \cdots)\subset \{ I_{\infty, t}\}_{t\in B}$ satisfying $F=J_\infty^{-1}\circ I_\infty$. 
This is an important point of the proof of Theorem \ref{thm:fibration}. 
\end{remark}

Gluing the above local result, we can prove Theorem \ref{thm:hariawase}. 

\begin{proof}[\indent \sc Proof of Theorem \ref{thm:hariawase}]
Since $f$ and $g$ are relatively compact Stein morphisms, for each point $t\in T$, we can take an open neighborhood $D$ of $t$ such that $f^{-1}(D)$ and $g^{-1}(D)$ are Stein. 
Then there exists a linear isomorphism $T_D: H^0(f^{-1}(D), mK_X) \to H^0(g^{-1}(D), mK_Y)$ which preserves the $L^{2/m}$-norm on each fiber. 
By Fubini's theorem, we see that $T_D$ induces a linear isomorphism 
$$
T_D: A(f^{-1}(D), mK_X)\to A(g^{-1}(D), mK_Y).
$$ 
Then $f^{-1}(D)$ and $g^{-1}(D)$ satisfy the assumptions of Theorem \ref{thm:fibration}. 
Hence, $n=l$. We can also construct the unique biholomorphism $F_D: f^{-1}(D) \to g^{-1}(D)$ satisfying the conditions of Theorem \ref{thm:fibration}. 

Then we define a global biholomorphic map $F:X \to Y$ as 
$$
F(z)=F_D(z)
$$
for some open neighborhood $D$ around $f(z)=:t$. The map $F$ is independent of the choice of $D$ for the following reason. 
Let $D_1$ and $D_2$ be two open neighborhoods such that $t\in D_1 \cap D_2$. 
Repeating the proof of Lemma \ref{lemma:fiberwise}, we have that $T_{(D_1)_t}=T_{(D_2)_t}$ since $T_{D_1}$ and $T_{D_2}$ commute with the restriction maps of sheaves. 
Therefore, fiberwise linear isometries $T_t :A(X_t, mK_{X_t})\to A(Y_t, mK_{Y_t})$ are uniquely induced by $T$. 
Then $F_{D_1}|_{X_t}$ and $F_{D_2}|_{X_t}$ are uniquely determined by $T_t$. 
In other words, 
from the results of Theorem \ref{thm:fibration}, $F_{D_1}$ and $F_{D_2}$ satisfy the following equations 
$$
|u(z)\wedge \overline{u(z)}|^{1/m}= |F_{(D_1)_{t}}^{\star}(T_tu)(z)\wedge \overline{F_{(D_1)_{t}}^{\star}(T_tu)(z)}|^{1/m}
$$
$$
|u(z)\wedge \overline{u(z)}|^{1/m}= |F_{(D_2)_{t}}^{\star}(T_tu)(z)\wedge \overline{F_{(D_2)_{t}}^{\star}(T_tu)(z)}|^{1/m}
$$
for $z\in X_t$ and $u\in A(X_t, mK_{X_t})$. Hence, 
$$
F_{D_1}(z)=F_{(D_1)_{t}}(z)=F_{(D_2)_{t}}(z)=F_{D_2}(z). 
$$
Then $F$ gives the biholomorphism $F:X \to Y$ which satisfies the condition of Theorem \ref{thm:hariawase}. 
\end{proof}


\section{Further remarks}\label{app:direct}
We introduce the direct proof of Corollary \ref{cor:domainfibration}. 
We do not need explicitly the $L^{2/m}$-variant of the Ohsawa-Takegoshi extension theorem. 

\begin{proof}[\indent \sc proof of Corollary \ref{cor:domainfibration}]
We obtain fiberwise linear isometries $T_t:A^p(X_t)\to A^p(Y_t)$ and biholomorphic maps $F_t : X_t \to Y_t$ for each $t\in B$ (cf. Lemma \ref{lemma:fiberwise}). 
Hence, we get $n=l$. 
Since $Y$ is a domain in $\mathbb{C}^{r+n}$, we take a global coordinate $(t_1, \cdots , t_r, w_{1}, \cdots , w_{n})$ on $Y$.  
We can assume that $g:Y \to B$ is defined as 
$$
g(t_1, \cdots , t_r, w_{1}, \cdots , w_{n})=(t_1, \cdots , t_r)
$$
without any loss of generality. Then $(w_1, \cdots, w_n)$ gives a global coordinate function on each $Y_t$. 
Let $\phi_0=T_t^{-1}(1), \phi_1 =T_t^{-1}(w_1), \cdots , \phi_n=T_t^{-1}(w_n)$ in $A^p(X_t)$. 
The results of \cite{DWZZ} (cf. Theorem \ref{thm:DWZZ}) imply that the expression of $F_t$ is given by 
$$
F_t = J_{\infty}^{-1}\circ I_\infty = J_n^{-1} \circ I_n=(\frac{\phi_1}{\phi_0}, \cdots , \frac{\phi_n}{\phi_0})
$$
globally on $X_t$. 

In this setting, we let $\Psi_0=1, \Psi_1=w_1 , \cdots , \Psi_n=w_n$ in $A^p(Y)$, and define $\Phi_0=T^{-1}(\Psi_0), \Phi_1=T^{-1}(\Psi_1), \cdots , \Phi_n=T^{-1}(\Psi_n)$ in $A^p(X)$. 
By the definition of $T_t$, we have 
$$
\Phi_0|_{X_t} = T_t^{-1}((T\Phi_0)|_{Y_t})=T_t^{-1}(1), \hspace{3.5mm}\Phi_j|_{X_t} = T_t^{-1}((T\Phi_j)|_{Y_t})=T_t^{-1}(w_j).
$$
Letting 
$$
F:=( t_1, \cdots, t_r, \frac{\Phi_1}{\Phi_0}, \cdots , \frac{\Phi_n}{\Phi_0} ), 
$$
we see that $F$ is a well-defined holomorphic map and $F|_{X_t}=F_t$. Hence, $F$ is a biholomorphic map from $X$ to $Y$ and 
gives the condition of Corollary \ref{cor:domainfibration}. 
\end{proof}

\end{document}